\documentclass[12pt,a4paper]{iopconfser}
\usepackage{amsmath,amssymb,amsthm,latexsym,amsbsy,euscript}

\newtheorem{theor}{Theorem}
\newtheorem{claim}[theor]{Claim}

\theoremstyle{definition}
\newtheorem{define}{Definition}

\newtheorem{state}[theor]{Proposition}
\newtheorem{prop}[theor]{Proposition}
\newtheorem{lem}[theor]{Lemma}
\newtheorem{conjecture}[theor]{Conjecture}
\newtheorem{open}{Open problem}
\newtheorem{ex}{Example}
\theoremstyle{remark}
\newtheorem{rem}{Remark}

\def\oldvec{\mathaccent "017E\relax }
\newcommand{\Or}{{\rm O\oldvec{r}}}

\newcommand{\BBR}{\mathbb{R}}
\newcommand{\BBN}{\mathbb{N}}

\newcommand{\cX}{{\EuScript X}}    

\newcommand{\ba}{{\boldsymbol{a}}}

\newcommand{\bx}{{\boldsymbol{x}}}

\newcommand{\veps}{\varepsilon}

\newcommand{\dd}{\partial}
\newcommand{\Id}{{\mathrm d}}

\newcommand{\lshad}{[\![}
\newcommand{\rshad}{]\!]}

\DeclareMathOperator{\Alt}{Alt}

\newcommand{\jour}[1]{\textit{{#1}}}
\newcommand{\vol}[1]{\textbf{{#1}}}

\begin{document}\pagestyle{plain}\large

\title{Kontsevich graphs act on Nambu\/--\/Poisson brackets,~I.\\ 
\mbox{ }\quad New identities for Jacobian determinants}

\author{Arthemy V Kiselev, Mollie S Jagoe Brown and Floor Schipper}

\affil{Bernoulli Institute for Mathematics, Computer Science and Artificial Intelligence, University of Groningen, P.O. Box~407, 9700 AK Groningen, The Netherlands}

\email{a.v.kiselev@rug.nl, m.s.jagoe.brown@gmail.com, f.m.schipper@rug.nl}

\begin{abstract}
Nambu\/-\/determinant brackets on~$\BBR^d\ni\bx=(x^1,\ldots,x^d)$,
$\{ f,g \}_d (\bx) = \varrho(\bx)\cdot\det\bigl( \dd(f,g,a_1,\ldots,a_{d-2})\bigr/\dd(x^1,\ldots,x^d) \bigr)$, with $a_i\in C^\infty(\BBR^d)$ and $\varrho\cdot\dd_{\bx}\in\mathfrak{X}^d(\BBR^d)$, are a class of degenerate (rank${}\leqslant 2$) Poisson structures with (non)\/linear coefficients, e.g., po\-ly\-no\-mi\-als of arbitrarily high degree.
With `good' cocycles in the graph complex, Kontsevich associated universal
--\,for all Poisson bi\/-\/vectors~$P$ on affine~$\BBR^d_{\text{aff}}$\,-- elements
$\dot{P} = Q^\gamma([P]) \in \mathrm{H}^2_{P}\bigl(\BBR^d_{\text{aff}}\bigr)$ in the Lich\-ne\-ro\-wicz\/--\/Poi\-s\-son second cohomology groups;
we note that known graph cocycles~$\gamma$ preserve the Nambu\/--\/Poisson class $\bigl\{ P(\varrho,[\ba]) \bigr\}$, and we express, directly from~$\gamma$, the evolution~$\dot{\varrho}$,\ $\dot{\ba}$ that induces~$\dot{P}$.

\mbox{ }\quad
Over all $d\geqslant 2$ at once, there is no `universal' mechanism for the bi\/-\/vector cocycles~$Q^\gamma_d$ to be trivial, $Q^\gamma_d = \lshad P,\smash{\vec{X}^\gamma_d}([P]) \rshad$, w.r.t.\ vector fields defined uniformly 
for all dimensions~$d$ 
by the same graph formula.
While 
over~$\BBR^2$, the graph flows $\dot{P} = Q^{\gamma_i}_{2D} \bigl( P(\varrho) \bigr)$ for $\gamma\in\{\gamma_3,\gamma_5,\gamma_7,\ldots\}$ are trivialised by vector fields $\smash{\vec{X}^{\gamma_i}_{2D}} = (\Id x\wedge\Id y)^{-1} \Id_{\text{dR}}\bigl(\text{Ham}^{\gamma_i}(P) \bigr)$ of peculiar shape,
we detect that in $d\geqslant 3$, the 
$1$-\/vectors from~2D, now with $P(\varrho$,$a_1$,$\ldots$,$a_{d-2}) $ inside, do not solve the problems $Q^{\gamma_i}_{d\geqslant 3} = \lshad P, \smash{\vec{X}^{\gamma_i}_{d\geqslant 3}} \bigl(P(\varrho,[\ba]) \bigr) \rshad$, yet they do yield a good Ansatz where we find solutions $\smash{\vec{X}^{\gamma_i}_{d=3,4}} \bigl(P(\varrho,[\ba]) \bigr)$.
In the study of the step $d\mapsto d+1$, by adapting the Kontsevich graph calculus to the Nambu\/--\/Poisson class of brackets, we discover more identities for the Jacobian determinants within $P(\varrho,[\ba])$, i.e.\ for multivector\/-\/valued $GL(d)$-\/invariants on~$\BBR^d_{\text{aff}}$.
\end{abstract}

\large
\noindent\textbf{Introduction
.}\quad This paper begins the series of three namesake texts which are devoted to deformations of Poisson brackets --\,by using Kontsevich's graph cocycles\,-- within the class of Nambu\/-\/determinant Poisson structures on~$\BBR^d$.
In the subsequent article (II.), see~\cite{MJB}, we establish the trivialisation of the tetrahedral $\gamma_3$-\/graph flow on the space of Nambu\/--\/Poisson brackets over~$\BBR^4$, that is in dimension four (cf.~\cite{skew23} for $d=3$ and~\cite{Ascona96,Anass2017} for $d=2$).
The uniqueness of trivialising vector fields~$\smash{\vec{X}^{\gamma_3}_d}$, themselves encoded by (generalisations of) Kontsevich's graphs built of (Nambu\/--)\/Poisson bi\/-\/vectors, modulo Poisson vector fields with Hamiltonians also expressed in terms of graphs, is verified in the third article (III.), see~\cite{FS}.

Now, in this paper we summarise newly observed properties of the Nambu\/--\/Poisson brackets. We discover five
classes of differential\/-\/polynomial identities which the Jacobian determinants (and the brackets derived from them) conjecturally satisfy;
all these hypotheses are open problems about (an effective description of) multivector\/-\/valued $GL(d)$-\/invariants, $2\leqslant d<\infty$, over affine spaces~$\BBR^d$.


\begin{define}\label{DefNambuPoissonCoord}
The Nambu\/-\/determinant Poisson bracket of $f,g\in C^\infty(\BBR^d)$ is expressed by the formula
\begin{equation}\label{EqNambuCoord}
\{ f,g \}_d(\bx) = \varrho(\bx) \cdot \det\smash{\Bigl(} \dd \bigl(f,g,a_1,\ldots,a_{d-2}\bigr)
\big/ \dd \bigl(x^1,\ldots,x^d \bigr) \smash{\Bigr)}, \qquad \bx\in\BBR^d,
\end{equation}
where $a_i\in C^\infty(\BBR^d)$ are Casimirs, $\varrho(\bx)\cdot \dd_{x^1}\wedge\ldots\wedge\dd_{x^d}$ is a $d$-\/vector, and $(x^i)$~are global coordinates on~$\BBR^d$, $d\geqslant 2$.
\end{define}

\begin{rem}\label{RemCartesian}
Kontsevich's construction of the graph cocycle action on the spaces of Poisson brackets is well defined over arbitrary finite\/-\/dimensional affine manifolds ${M}_{\text{aff}}^{d<\infty}$. In our present illustration of this concept and in our study of the action specifically upon the class of Nambu\/--\/Poisson brackets, we take $M\mathrel{{:}{=}}\mathbb{R}^d$ with natural global Cartesian coordinates, e.g., denoted by $x,y,z,w$ on~$\mathbb{R}^4$. Yet of course, the term `Cartesian' serves here as the marker for an atlas of affine coordinate charts on ${M}_{\text{aff}}^d$, that is all the coordinate tuples which are obtained from a given one by using affine reparametrisations.
\end{rem}

\begin{rem}\label{RemNambuNonlin}
Over affine manifolds ${M}_{\text{aff}}^d$, the degree of polynomial functions is well defined; beyond scalar functions, this is also true in particular for the components $P^{ij}(\boldsymbol{x})$ of the Poisson tensor $P$ (provided these components are polynomial in every chart of some cover for an orientable manifold ${M}_{\text{aff}}^d$). Therefore, Nambu's formula $P=\lshad \ldots \lshad\varrho\,\partial_{\bx},a_1\rshad, \ldots a_{d-2}\rshad$ of Poisson structures $P(\varrho,[\boldsymbol{a}])$ on orientable ${M}_{\text{aff}}^d$ offers us the brackets with coefficients of arbitrarily high polynomial degree, which is achieved by taking polynomial scalar functions $a_i$ and taking the $d$-vector $\varrho(\boldsymbol{x})\cdot\partial_{\boldsymbol{x}}$ (again, a tensor) with polynomial coefficient $\varrho$ of suitable degrees. 

Let us remember also that the symplectic leaves of the Nambu--Poisson structures $P(\varrho,[\boldsymbol{a}])$ are at most of dimension two. Indeed, the leaves are selected by intersecting the level sets of the $d-2$ Casimirs. (The Euler linear bracket on $\mathfrak{so}(3)^*$, in Cartesian coordinates described by $\{x,y\}=z$ and so on cyclically, foliates $\mathbb{R}^3$ by the concentric spheres $\{(x,y,z) | a=\tfrac{1}{2}(x^2+y^2+z^2)=\tfrac{1}{2}r^2\geq0\},$ providing a typical example: at $r=0$, the zero-dimensional symplectic leaf amounts to the central point of all spheres.)
\end{rem}

\begin{rem}\label{RemNaryNambu}
Not only does the binary bracket $\{{\cdot},{\cdot}\}_d$ satisfy the Jacobi identity but also does the $N$-ary bracket,
\[
\{ f_1,\ldots,f_N\}_d(\bx) = \varrho(\bx)\cdot \det \smash{\Bigl(} \dd \bigl( f_1,\ldots,f_N, a_{N-1},\ldots,a_{d-2} \bigr) \big/ \dd \bigl( x^1,\ldots,x^d \bigr) \smash{\Bigr)},
\]
read off literally from Eq.~\eqref{EqNambuCoord} for $2\leqslant N\leqslant d$, satisfy one of the many possible $N$-ary generalizations of the Jacobi identity, `\textsl{the adjoint action is a derivation of the bracket}' (see~\cite{Nambu1973,Takhtajan1994CMP}
):
\begin{multline*}
\{f_1,\ldots,f_{N-1}, \{g_1,\ldots,g_N\}_d \}_d =
\{ \{f_1,\ldots,f_{N-1},g_1\}_d, g_2,\ldots,g_N \}_d + {} \\
\{g_1,\{f_1,\ldots,f_{N-1},g_2\}_d, g_3,\ldots,g_N \}_d + \cdots +
\{ g_1,\ldots,g_{N-1}, \{ f_1,\ldots,f_{N-1},g_N\}_d \}_d.
\end{multline*}
Let us remember that at either $N=2$ (Poisson case) or $N>2$, the Jacobi identities are quadratic in the $N$-ary structure.
\end{rem}

This paper is structured as follows.
In~\S\ref{SecMKgraphs} we recall from~\cite{Ascona96} Kontsevich's idea of acting --\,by suitable nontrivial graph cocycles~$\gamma$\,-- on the spaces of all Poisson bi\/-\/vectors~$P$ on affine manifolds of arbitrary finite dimension~$d$. We note that for the wheel\/-\/cocycle generators $\gamma_3$,\ $\gamma_5$,\ $\gamma_7$,\ $\gamma_9$,\ $\ldots$ of the Grothendieck\/--\/Teichm\"uller Lie algebra~$\mathfrak{grt}$ (see~\cite{TWGRT2015} and~\cite{JNMP2017}), the corresponding $2$-\/cocycles $\dot{P} = Q^{2\ell+1}(P)\in\ker\lshad P,{\cdot}\rshad$ are \textsl{not} trivialised by any vector fields~$\smash{\vec{X}^{\gamma_{2\ell+1}}}$ which would again be encoded by graphs and therefore, solve the trivialisation problems $Q^{\gamma} = \lshad P,\smash{\vec{X}^\gamma}\rshad$ universally over all dimensions~$d\geqslant 2$.

In~\S\ref{SecNambuMicroG} we adapt Kontsevich's graph approach to the differential calculus of multi\/-\/vectors on~$\BBR^d_{\text{aff}}$ of unspecified dimension~$d$ --- now, to the dimension\/-\/specific classes of Nambu\/-\/determinant Poisson bi\/-\/vectors $P(\varrho,[\ba])$ on affine~$\BBR^{d\geqslant 3}$. We thus work with the \textsl{Nambu micro\/-\/graphs} (see~\cite{skew23}), in which the top\/-\/degree $d$-\/vector $\varrho\cdot\dd_{\bx}$ is resolved against the Casimir(s) $\ba=(a_1$,\ $\ldots$,\ $a_{d-2})$ in each subgraph that encodes a copy of the bi\/-\/vector~$P(\varrho,[\ba])$.
We give examples of $(k\geqslant 0)$-\/vector Nambu micro\/-\/graphs which do not `equal minus themselves thanks to an automorphism' but which nevertheless encode identically vanishing $k$-\/vectors. We observe that the $(d+1)$-\/dimensional \textsl{embedding} of a Nambu micro\/-\/graph which vanished in dimension~$d$ still vanishes in dimension~$d+1$. Likewise, we see that for \textsl{synonyms}, i.e.\ for topologically nonisomorphic $(k\geqslant 0)$-\/vector Nambu micro\/-\/graphs which encode identically equal $k$-\/vectors on~$\BBR^d$, their graph embeddings over dimension $d+1$ still are synonyms; the same is conjecturally true for longer nontrivial linear combinations of micro\/-\/graph formulas: if a formula $\phi(\sum_i c_i\Gamma_i)=0 \in \mathfrak{X}^k\bigl(\BBR^d_{\text{aff}}\bigr)$ is a nontrivial relation and if $\Gamma\hookrightarrow\smash{\widehat{\Gamma}}$ is the embedding of micro\/-\/graphs, then the formula $\phi(\sum_i c_i\smash{\widehat{\Gamma}_i})=0 \in \mathfrak{X}^k\bigl(\BBR^{d+1}_{\text{aff}}\bigr)$ remains a valid relation. We seek to understand the mechanism of this preservation, under $d\mapsto d+1$, of relations for this class of multivector\/-\/valued $GL(d)$-\/invariants.

In the second part of this paper, starting in~\S\ref{SecFlows2D} we state the facts of trivialisation for Kontsevich's graph flows, $\dot{P}_{2D} = Q^\gamma_{2D}\bigl(P(\varrho)\bigr)$, over dimension $d=2$ for the $(2\ell+1)$-\/wheel 
graph cocycles $\gamma\in\bigl\{\gamma_3,\gamma_5,\gamma_7\bigr\}$ and for the Lie bracket $[\gamma_3,\gamma_5]$. We notice that not only for the tetrahedron~$\gamma_3$ but also for the larger graph cocycles, the affine spaces of trivialising vector fields $\smash{\vec{X}^\gamma_{2D}}\bigl(P(\varrho)\bigr)$ that solve $Q^\gamma_{2D}\bigl(P(\varrho)\bigr) = \lshad P, \smash{\vec{X}^\gamma_{2D}} \rshad$ do contain a Hamiltonian vector field
$\smash{\vec{X}^\gamma_{2D}} = (\Id x\wedge\Id y)^{-1} \bigl(\Id_{\text{dR}} \text{Ham}^\gamma\bigl(P(\varrho) \bigr) \bigr)$ given by the canonical symplectic structure $\omega_2 = \Id x\wedge\Id y$ on~$\BBR^2$ and by Hamiltonians $\text{Ham}^\gamma$ which, for every such graph cocycle~$\gamma$, are encoded by graphs built of wedges.
(In the subsequent papers~\cite{MJB,FS} we discover that in dimensions $d=3$ and~$4$,
solutions $\smash{\vec{X}^{\gamma_3}_{d=3,4}}$ appear over the Ansatz of linear combinations of micro\/-\/graph \textsl{descendants} of the $1$-\/vector graphs in such particular solution $\smash{\vec{X}^{\gamma_3}_{2D}} = \omega_2^{-1}\bigl(\Id_{\text{dR}} \text{Ham}^{\gamma_3} \bigr)$ --- unlike 
for most of the other graph pairs that encode the solution in dimension two.)

Finally, in~\S\ref{SecFlowsHighDim} we conjecture the formulas of velocities~$\dot{\varrho}$ and~$\dot{a}_i$, expressed directly in terms of the graph cocycles~$\gamma$, that imply (by the Leibniz rule) 
the evolution $\dot{P} = Q^\gamma\bigl(P(\varrho,[\ba]) \bigr)$ of the Nambu\/--\/Poisson structures.
By contrasting the antisymmetry of the Nambu\/-\/determinant brackets $P(\varrho,[\ba])$ w.r.t.\ the flips $a_i\mapsto -a_i$ for all $d\geqslant 3$ and w.r.t.\ permutations of the set of Casimirs~$a_i$ for $d\geqslant 4$ against the symmetry of their $\gamma$-\/flows $Q^\gamma_d(P)$, we motivate the existence of trivialisation, $Q^{\gamma_{2\ell+1}}_d = \lshad P, \smash{\vec{X}^{\gamma_{2\ell+1}}_d} \rshad$, for the $(2\ell+1)$-\/wheel graph cocycles~$\gamma_{2\ell+1}$ (and for their iterated commutators on even number of vertices, cf.~\cite{JNMP2017,TWGRT2015}).


\section{Preliminaries: Kontsevich graph cocycles act on Poisson brackets}\label{SecMKgraphs}
In the seminal paper~\cite{Ascona96}, Kontsevich introduced the graph complex action on the spaces of multivectors on affine finite\/-\/dimensional manifolds. We recall that real vector spaces of undirected finite graphs with a global ordering of edges (First${}\wedge\ldots\wedge{}$Last, quotient over the relations Edge${}_i\wedge{}$Edge${}_j=-{}$Edge${}_j\wedge{}$Edge${}_i$) are endowed with the structure of differential graded Lie (super-)\/algebra (dgLa), namely with the Lie bracket $[\cdot,\cdot]$ from the graded commutator of graph insertion into vertices and with the vertex blow-up differential $\sf{d}=[\bullet\!\text{\textbf{--}}\!\bullet,{\cdot}]$; we refer to~\cite{JNMP2017,OrMorphism2018} and references therein for all definitions and details. Graph cocycles on $n$ vertices and $2n-2$ edges are of particular interest because it is in this vertex-edge bi-grading where graph cocycles $\gamma$ can act -- by non-identically zero shifts $P\mapsto P+\varepsilon Q^{\gamma}+\overline{o}(\varepsilon)$ and in a possibly nontrivial way, $Q^{\gamma}\neq\lshad P,\vec{X}^{\gamma}\rshad$ -- on the spaces of Poisson bi-vectors $P$ over the affine manifold ${M}_{\text{aff}}^d$ at hand. Willwacher established the existence of at least countably many such cocycles (see~\cite{TWGRT2015}): the $(2\ell+1)$-wheel graph cocycles $\gamma_{2\ell+1}$, $\ell\in\mathbb{N}$, stem from the generators of Drinfeld's Grothendieck-Teichmüller Lie algebra $\mathfrak{grt}$, so that their iterated commutators stay in the good vertex-edge bi-grading and remain non-trivial cocycles in the graph complex.

\begin{ex}\label{ExGamma3579}
The smallest nontrivial graph cocycle, on $n=4$ vertices and $2n-2=6$ edges, is the 3-wheel itself: it is the tetrahedron $\gamma_3$ (the full graph on 4 vertices); it appeared already in \cite{Ascona96}. The pentagon-wheel graph cocycle $\gamma_5$, consisting of two graphs on 6 vertices and 10 edges (see Table~\ref{TabNumberGraphsInCocycles} on p.~\pageref{TabNumberGraphsInCocycles} below), was known to Kontsevich and to Willwacher; the cocycle $\gamma_5$ is described in \cite{JNMP2017}. The heptagon-wheel graph cocycle $[\gamma_7]$ was obtained in \cite{JNMP2017}; now, the space of graphs on 8 vertices and 14 edges is big enough to provide degree(s) of freedom in the cohomology class $[\gamma_7]$ due to the now-possible coboundaries $\textsf{d}(\beta)$ from graphs $\beta$ on 7 vertices and 13 edges; the shortest-known representative $\gamma_7$ of the nontrivial cohomology class $[\gamma_7]$ is a linear combination of 46 graphs. The next graph cocycle, in the vertex-edge bi-grading (9,16) immediately following (8,14) along the ray ($n,2n-2$), is represented by the commutator $[\gamma_3,\gamma_5]$; its encoding is worked out in~\cite{BuringPhD}. At the ISQS28 conference (CVUT Prague, 1--5 July 2024), R.~Buring reported a representative $\gamma_9$ of the 9-wheel graph cocycle on 10 vertices and 18 edges in each of its 13,723~terms. (As the vertex number grows, the ($n,2n-2$)-homogeneous component of the graph space can contain not just one but many nontrivial graph cocycles which, modulo the coboundaries, are linearly independent.) Let us remember also that each of these good graph cocycles $\gamma_3,\ldots,\gamma_9$ was obtained `anew', i.e.\ not --\,by following Willwacher's isomorphism\,-- from the generators of the Lie algebra~$\mathfrak{grt}$; the task of writing explicit formulas for the correspondence between $\mathfrak{grt}$~and representatives of the classes $[\gamma_{2\ell+1}]$, and of their iterated commutators, is a work in progress (M.~Kontsevich, private communication). 
\end{ex}

Graphs $\gamma$ with a global edge ordering $E(\gamma)$ are mapped to endomorphisms of the space of multivectors on ${M}_{\text{aff}}^d$ by the edge orientation morphism (see \cite{Ascona96} and \cite{f16,OrMorphism2018,sqs17}). Every directed edge $\vec{e}$, decorated with a summation index $i_e$ which runs from 1 to $d$, denotes the derivation $\partial/\partial x^{i_e}$ of the (multi)vector contained in the arrowhead vertex; the local exterior ordering of the outgoing edges, $\vec{e}_1\wedge\vec{e}_2\wedge \ldots \wedge\vec{e}_k$, which thus expresses the skew-symmetry of the $k$-vector (in the arrowtail vertex) with respect to its arguments, is inherited at every vertex from the global ordering of edges in the initially taken graph $\gamma$, where $E(\gamma)=...\wedge e_1\wedge e_2\wedge ...\wedge e_k\wedge ...$. In our present study of the graph complex action on Poisson brackets, it suffices to enlarge the graph $\gamma$ by two sink vertices and to consider only those portraits of edge direction where the new graphs, with exactly one arrow directed to either of the sinks, are built entirely of wedges $\smash{\xleftarrow{i}\!\!\bullet\!\!\xrightarrow{j}}$ (for the Poisson bi-vectors $P=(P^{ij})$ which are the building blocks).

\begin{define}\label{DefMKgraph}
    Directed graphs built over $m\geqslant 0$ sinks from $n\geqslant 1$ wedges (with local ordering Left${}\prec{}$Right for the two outgoing arrows at every wedge top) are called the \textsl{Kontsevich (di)graphs}; note that 1-loops (\textsl{tadpoles}) are allowed, although the $(2\ell+1)$-wheel graph cocycles (and their commutators) stemming from $\mathfrak{grt}$ always admit representatives $\gamma_{2\ell+1}$ (resp., $[\ldots[\gamma_{2\ell+1},\gamma_{2p+1}]\ldots]$) without 1-loops. 
\end{define}

\begin{ex}\label{ExTetra}
    By directing the four edges in the tetrahedron $\gamma_3$ in such a way that the vertices of $\gamma_3$ are the four wedge tops and the two excessive edges are sent to the two new sinks $\overline{0}$ and $\overline{1}$, we obtain --\,with multiplicities 8 and 24\,-- two topologically non-isomorphic pictures (see \cite{f16}): one is already skew over the sinks and the other, to give a bi-vector, is skew-symmetrised; this yiels $1+2=3$ Kontsevich graphs. Taken with their multiplicities $8:24=1:3$, they encode the bi-vector $Q^{\gamma_3}([P])= \Or (P\otimes P\otimes P\otimes P)$. Likewise, for the pentagon-wheel graph cocycle $\gamma_5$, we obtain the 91 bi-vectors realised by Kontsevich graphs (\cite{sqs17}), and so on (see \cite{OrMorphism2018,BuringPhD} and Table~\ref{TabNumberGraphsInCocycles}).
\begin{table}[htb]
\caption{\label{TabNumberGraphsInCocycles}
The number of (un)directed graphs in the graph cocycles~$\gamma$ and Poisson cocycles~$Q^\gamma$.}
\begin{center}
\begin{tabular}{|l|r|r|r|r|r|}
\hline
Cocycle~$\gamma$: & $\gamma_3$ & $\gamma_5$ & $\gamma_7\in[\gamma_7]$ & $[\gamma_3,\gamma_5]$ & $\gamma_9\in[\gamma_9]$ \\
\hline
\#\,vertices: & 4 & 6 & 8 & 9 & 10 \\
\#\,edges: & 6 & 10 & 14 & 16 & 18 \\
\#\,graphs in $\gamma$: & 1 & 2 & 46 & 68 & 13,723 \\
\#\,bi\/-\/vectors in $Q^\gamma$: & 2 & 91 & 20,422 & 42,252 & ? \\
\#\,directed graphs in $Q^\gamma$: & 3 & 167 & 37,185 & ? & ? \\
\hline
\end{tabular}
\end{center}
\end{table}
\end{ex}

\begin{lem}[{see \cite{Ascona96} and \cite{OrMorphism2018,TWGRT2015}}]\label{LemmaCocycleToCocycle}
    Whenever $\gamma\in\ker\textsf{d}$ is a nontrivial graph cocycle over $n$ vertices and $2n-2$ edges, and $P$ is a Poisson bi-vector on an affine manifold ${M}_{\text{aff}}^d$, the bi-vector $Q^{\gamma}([P])\mathrel{{:}{=}} \Or(P^{\otimes^n})$ is a Poisson 2-cocycle: $Q^{\gamma}([P])\in\ker\lshad P,\cdot\rshad$.
\end{lem}

\begin{claim}\label{ThNoTrivViaGraphs}
    Over all affine Poisson manifolds (${M}_{\text{\textup{aff}}}^{d<\infty},P$) at once, the $(2\ell+1)$-wheel graph cocycle deformations $\dot{P}=Q^{\gamma_{2\ell+1}}(\smash{P^{\otimes^{2\ell+2}}} )$ cannot be Poisson coboundaries `universally' over $d\geqslant 2$ with respect to always the same linear combinations $X^{\gamma_{2\ell+1}}$ of Kontsevich 1-vector graphs built of $n=2\ell+1$ wedges. Specifically, there is no solution~$\Diamond$ --\,at the level of Formality graphs from~\textup{\cite{MK97}}\,-- to the equation 
\[    
Q^{\gamma_{2\ell+1}} - \lshad P,\text{ any 1-vector graphs on }2\ell+1\text{ wedges }\rshad=\Diamond\bigl(P,\tfrac{1}{2}\lshad P,P\rshad\bigr),
\]
where the right-hand side encodes bi-vectors that vanish by force of the Jacobi identity for the Poisson structure $P$.
\end{claim}



\begin{proof}[Sketch of the proof]
Tadpoles are neither produced nor destroyed by the differential calculus of graphs (when the Jacobiator is expanded by definition and when an arrow works over the vertices of a (sub)graph by the Leibniz rule, e.g., during the calculation of the Schouten bracket $\lshad{\cdot},{\cdot}\rshad$).
Therefore, the linear problem of $\gamma_{2\ell+1}$-\/deformation's trivialisation at the level of Formality graphs is filtered by the number of tadpoles in a graph.

We recall that by construction, there are no tadpoles in the 
inhomogeneity, 
$\dot{P} = \Or(\gamma_{2\ell+1})\bigl(P^{\otimes^{2\ell+2}}\bigr)$.
To establish the absence of universal trivialisation, it 
suffices to inspect the $0$th layer of the problem 
with Formality graphs 
{without tadpoles}; here, the obstruction is easily attained at all~$\ell\in\BBN$.
\end{proof}




Let us remember that over every affine Poisson manifold ${M}_{\text{aff}}^{d<\infty}$ of any finite dimension $d\geqslant 2$, each Kontsevich graph gives us a well\/-\/defined $k$-vector (that belongs --\,possibly, after due antisymmetrisation\,-- to the space $\mathfrak{X}^k({M}_{\text{aff}}^d)$); the formula of that $k$-vector behaves well under affine coordinate reparametrisations: the shifts are not felt at all, whereas the linear transformations from $GL(d<\infty)$ are absorbed by the reparametrised copies of the Poisson tensor in the vertices of Kontsevich graphs. Yet it does occur that topologically nonisomorphic Kontsevich graphs of equal arity (e.g., 1-vectors) and with equal number of vertices (hence of equal polynomial degree in the coefficients of $P=(P^{ij})$ or their derivatives) encode linearly dependent $k$-vector formulas in a given dimension $d<\infty$. That is, the projections of universally defined $GL(\infty)$-invariants (encoded by Kontsevich graphs) to $GL(d)$-invariants become constrained by linear relations. 

\begin{ex}[{see Claim~2 in~\cite{FS}}]\label{ExSynonymsOneVect}
The 14 admissible non-isomorphic 1-vector Kontsevich graphs built of 3 wedges over one sink evaluate, in dimension $d=2$, to only three linearly independent formulas of vector fields on~$\mathbb{R}^2$.
\end{ex}

In what follows, by evaluating Kontsevich \textsl{nonzero} graphs to the respective formulas in finite dimensions, we shall encounter (\textit{i}) instant vanishings: $\phi(\Gamma)\equiv0$ for $\Gamma\neq-\Gamma$, for a single graph $\Gamma$; (\textit{ii}) \textsl{longer} linear relations that involve three or more graphs. By construction, these identities are dimension\/-\/dependent: besides, identities can be specific to the Nambu--Poisson class of bi-vectors over dimension $d<\infty$, that is not hold for arbitrary Poisson bi-vectors~$P$.

\section{Basic concept: Nambu micro\/-\/graphs over~$\mathbb{R}^d_{\text{\textmd{aff}}}$}\label{SecNambuMicroG}

\begin{define}\label{DefNambuMicroGraph}
The \textsl{Nambu graph} over dimension $d$ (here $3\leqslant d <\infty$) is the directed graph consisting of the source vertex (containing the $d$-vector coefficient $\varrho(\boldsymbol{x})\cdot\varepsilon^{i_1 ...i_d}$) from which run $d$ arrows (decorated with the summation indices $i_1,\ldots,i_d$); by convention, the 3rd, $\ldots$, $d$th arrows head to the terminal vertices with the respective Casimirs $a_1,\ldots,a_{d-2}$, whereas the 1st and 2nd arrow, ordered Left${}\prec{}$Right as usual, encode the derivations of the arguments of the Nambu\/--\/Poisson bi-vector $P(\varrho,[\boldsymbol{a}])$ from Eq.~\eqref{EqNambuCoord}. From the definition of the Levi-Civita symbol $\varepsilon^{i_1\ldots i_d}$ it follows that the $d$-tuple of outgoing arrows is wedge\/-\/ordered: a swap of any two arrows reverses the sign in front of the Nambu graph. 

Nambu graphs, each realising a copy of Nambu\/--\/Poisson bracket~\eqref{EqNambuCoord}, are the building blocks (i.e.\ subgraphs) in the \textsl{Nambu micro\/-\/graphs} over $m\geqslant 0$ sinks.\footnote{\label{FootNambuTadpoles}%
We consider only finite Nambu micro\/-\/graphs; note also that 1-loops are allowed in Nambu micro-graphs.}
\end{define}

\begin{ex}\label{ExNambuMicroG}
Let $d=3$; let 0,1,2 be the sinks, 3 and 4 be the Levi-Civita vertices, and 5,6 be the Casimir vertices. Then the digraph\footnote{\label{FootNotationNambuG}%
We list the target vertices of the ordered $d$-tuples of arrows issued from the Levi-Civita vertices, themselves ordered by a given vertex labelling.} 
$\Gamma_1=[0,1,5;2,5,6]$ is a Nambu micro-graph.\\[0.5pt]
\mbox{ }$\bullet$\quad  Let $d=3$; let 0 and 1 be the sinks, 2 and 3 be the Levi-Civita vertices, and 4,5 be the Casimirs. Then the digraph $\Gamma_2=[0,1,4;3,4,5]$ is a Nambu micro-graph (with a 1-loop on vertex 3).\\[0.5pt]
\mbox{ }$\bullet$\quad  Let $d=3$; let 0 and 1 be the sinks, 2 and 3 be the Levi-Civita vertices, and 4,5 be the Casimirs. Then the digraph $\Gamma_3=[0,1,4;2,4,5]$ is a Nambu micro-graph.\\[0.5pt]
\mbox{ }$\bullet$\quad  Let $d=3$; let 1 and 2 be the Levi-Civita vertices and 3,4 be the Casimirs; then the digraph $\Gamma_4=[1,2,4;1,2,4]$ is \textsl{not} a Nambu micro-graph (because it is not built from the Nambu (sub)graphs: its vertices and edges are not organised into a union of whole copies of the Nambu\/--\/Poisson structure over~$d=3$.
\end{ex}

\begin{rem}\label{RemNambuFromMK}
    Whenever the Poisson bracket at hand is Nambu (from Eq.~\eqref{EqNambuCoord}), linear combinations of Nambu micro-graphs can be obtained by magnifying the internal vertices of a Kontsevich graph under a microscope that resolves the elements $\varrho$ against each of the Casimirs $a_1,\ldots,a_{d-2}$ in the Nambu\/--\/Poisson bi-vector. Every arrow which hit $P$ in the Kontsevich graph now works over the elements of $P(\varrho,[\boldsymbol{a}])$ by the Leibniz rule. The Left${}\prec{}$Right ordering of the edge pairs from every wedge $\smash{\xleftarrow{L}\!\!\bullet\!\!\xrightarrow{R}}$ for~$P$ is now inherited by the 1st and 2nd arrows in the $d$-tuple issued from the respective Levi-Civita vertex. However, not all Nambu micro-graphs are obtained by such Leibniz rule expansions (in particular, when some of the terms from these expansions are omitted -- but not only then).
\end{rem}

\begin{define}\label{DefMKmicroG}
The \textsl{Kontsevich micro\/-\/graph} over dimension $d$ is the (linear combination of) Nambu micro-graph(s) which is obtained from (a linear combinations of) Kontsevich's graphs by postulating the bi-vector~$P$ to be Nambu\/--\/Poisson, $P=P(\varrho,[\boldsymbol{a}])$ over $d\geqslant 3$, and then by working out all the Leibniz rules for each of the edges which acted on the vertices containing~$P$ in the originally taken Kontsevich graph(s).
\end{define}

\begin{ex}\label{ExMKmicroG}
Let $d=3$; let 0,1 be the sinks, 2 and 3 be the Levi-Civita vertices, and 4,5 be the Casimirs. Then the sum of digraphs $[0,1,4;2,3,5]+[0,1,4;4,3,5]$ is a Kontsevich micro-graph.\\[0.5pt]
\mbox{ }$\bullet$\quad  But $\Gamma_3\neq0$ from Example~\ref{ExNambuMicroG} is \textsl{not} a Kontsevich micro-graph --- because if it were, it would be obtained from a Kontsevich graph with a double edge; that Kontsevich graph would therefore be \textsl{zero}, i.e.\ equal to minus itself, whereas $\Gamma_3\neq0$ over~$d=3$.
\end{ex}

\begin{prop}\label{PropExistNonzeroVanishingG}
There exist nonzero but still vanishing (micro-)\/graphs.
\end{prop}

\begin{ex}\label{ExNonzeroVanishing}
There are twelve vanishing Nambu micro-graphs (of them, three are zero and nine nonzero) within the set of 41 Nambu 1-vector micro-graphs, built of three Nambu (sub)graphs, which show up in the Kontsevich micro-graph expansion over $d=3$ of the two `sunflower' graphs $\Gamma',\Gamma''$, see Eq.~\eqref{EqSunflower} below, 
whose linear combination $X_{2D}^{\gamma_3}$ sufficed to trivialise the tetrahedral $\gamma_3$-flow on the space of (Poisson) bi-vectors in dimension two (cf.~\cite{Ascona96} and \cite{Anass2017,f16,skew21,skew23}, also~\cite{MJB,FS}); now over $d=3$, these twelve (non)zero vanishing micro-graphs are listed in \cite[Lemma~2]{MJB}.\\[0.5pt]
\mbox{ }$\bullet$\quad  Again, among the 324 one-vector nambu micro-graphs which show up in the Kontsevich micro-graph expansion --\,now over $d=4$\,-- of the `sunflower' graphs, there are 54 vanishing micro-graphs (see the Appendix in \cite{MJB}).\\[0.5pt]
\mbox{ }$\bullet$\quad  Among the 21 Hamiltonians (i.e.\ 0-vector Nambu micro-graphs, without sinks) built of two Nambu structures over dimension $d=4$, there is a unique nonzero vanishing graph $H_{d=4}^{\equiv0}=[1,2,3,5;3,4,5,6]$ (here 1,2 are the Levi-Civita vertices, 3 and 4 are the Casimirs $a_1$, and 5,6 are the Casimirs $a_2$, 
see~\cite[Lemma~16]{FS}). --- In lower dimensions $d=2,3$, there are no vanishing Hamiltonians built of two (Nambu--)\/Poisson structures.
\end{ex}

\begin{define}\label{DefEmbed}
Consider a (micro-)\/graph~$\Gamma$ built from Nambu\/-\/Poisson bi\/-\/vector subgraphs over~$\BBR^d$, with copies of $\varrho\cdot\veps^{\vec{\imath}}$ and `their own' Casimirs $a_1$,\ $\ldots$,\ $a_{d-2}$ in different vertices.
Now over~$\BBR^{d+1}$, let every Levi\/-\/Civita vertex $\varrho\cdot\veps^{i_1\ldots i_{d+1}}$ send a new arrow to a new terminal vertex (with `Levi\/-\/Civita's 
own' new Casimir~$a_{d-1}$) of in\/-\/degree${}\equiv 1$; that is we \textsl{embed} $\Gamma\hookrightarrow\smash{\widehat{\Gamma}}$ such that no Leibniz rules are reworked.
\end{define}

Note that in the resulting micro\/-\/graph $\smash{\widehat{\Gamma}}$ with edges decorated by summation indices, the value $d+1$ of the index on every new edge reproduces the formula of~$\Gamma$ times $\bigl(\dd a_{d-1} / \dd x^{d+1} \bigr)^p$, with the power $p=\#\varrho$ in~$\Gamma$ --- yet, in the course of summation, there appear cross\/-\/terms with $\dd a_{d-1} / \dd x^i$ with~$i\leqslant d$.

\begin{state}\label{PropEmbed}
The only vanishing Hamiltonian $H^{\equiv0}_{d=4}(P\otimes P)=0$ over~$\BBR^4$, when embedded into dimension five, remains vanishing: $\smash{\widehat{H}}{}^{\equiv0}_{d=5}(P\otimes P)=0$.\\[0.5pt]
\mbox{ }$\bullet$\quad  The embedding into dimension four remains vanishing for each of the twelve vanishing $1$-vector descendants (in dimension three) of the two `sunflower' graphs~$\Gamma'$,\ $\Gamma''$ from
~\eqref{EqSunflower}.
\end{state}

\begin{define}[{cf.\ Definition~4 in~\S6 from~\cite{FS}}]\label{DefSynonyms}
Two topologically nonisomorphic graphs $\Gamma_1 \neq \Gamma_2$ are called \textsl{synonyms} if $\phi(\Gamma_1) = c \cdot \phi(\Gamma_2)$ with $c\in\BBR\setminus\{0\}$, that is, the two graphs provide the same multivector up to a nonzero constant.
\end{define}

\begin{ex}\label{ExSynonyms}
Over $d=3$, consider the seven nonisomorphic 0-vector Nambu micro-graphs (i.e.\  Hamiltonians) built of two Nambu (sub)\/graphs. A pair and a triple of synonyms are displayed in~\cite[Eq.~(4) and Lemma~11]{FS}; the remaining four formulas obtained from those seven graphs are linearly independent.\\[0.5pt]
\mbox{ }$\bullet$\quad  Likewise, over $d=4$, the 21 non-isomorphic Hamiltonians on two Nambu sub-graphs contain 8 pairs of synonyms (and one vanishing micro-graph): 
see~ \cite[Eq.~(5) and Lemma~16]{FS}.
\end{ex}

\begin{state}\label{PropEmbedPreservesSynonyms}
For the seven and four synonyms of $1$-vector graphs $\Gamma'$ and~$\Gamma''$ in the `sunflower' $\cX^{\gamma_3}_{d=2}=\Gamma'+2\cdot\Gamma''$, the embedding of every linear relation $\Gamma'_\alpha = \Gamma'_\beta$ or $\Gamma''_r = \Gamma''_s$ (for their formulas in dimension two) into higher dimensions $d=3$ and $d=4$ remains a valid linear relation between the formulas of larger micro\/-\/graphs:
$\smash{\widehat{\Gamma'_\alpha}} = \smash{\widehat{\Gamma'_\beta}}$ and $\smash{\widehat{\Gamma''_r}} = \smash{\widehat{\Gamma''_s}}$.
\end{state}


\begin{open}\label{OpenPrbEmbedLinRel}
Is it true that the embedding of 
Nambu micro\/-\/graphs always preserves linear relations between their respective formulas\,?
\end{open}

\section{If Kontsevich's flows over 2D are coboundaries, then which ones\,?}\label{SecFlows2D}

Over dimension $d=2$, every bi-vector $P = \varrho(x,y)\, \dd_x\wedge\dd_y$ is Poisson (in absence of nonzero tri-vectors $\tfrac{1}{2}\lshad P,P\rshad$ for the left-hand side of the Jacobi identity). For the same reason, every bi-vector is a Poisson 2-cocycle. Yet the graph cocycle flows at hand are not obliged to be coboundaries because the Lichnerowicz-Poisson second cohomology does not vanish \textit{a priori} over $d=2$. Indeed, the structure $P$ can degenerate on a locus inside ${M}_{\text{aff}}^2$, so that nontrivial Poisson cocycles start to exist.\footnote{For example, take $\varrho(x,y):=x^py^q\cdot\varrho(x,y)$, where $p,q\gg1$ and $\varrho$ is smooth near the origin of $\mathbb{R}^2\ni(x,y)$. Then every coboundary $\lshad P,\vec{X}(x,y)\rshad$ also vanishes at (0,0) for all smooth vector fields $\vec{X}$ on $\mathbb{R}^2$, still there exist many bi-vectors $Q\in\mathfrak{X}^2(\mathbb{R}^2)$, hence $Q\in\ker\lshad P,\cdot\rshad$, which do not vanish at (0,0), so these $Q\not\in\text{im}\lshad P,\cdot\rshad$ mark nontrivial Poisson 2-cocycles.}

We recall from Claim~\ref{ThNoTrivViaGraphs} that no universal --\,at the level of Kontsevich graphs\,-- trivialisation can be possible over $d\geqslant 2$ for the $\gamma_{2\ell+1}$-wheel graph flows $\dot{P}=\Or(\gamma_{2\ell+1})(P^{\otimes^{2\ell+2}})$. It is now all the more amazing that not only are these $\gamma_{2\ell+1}$-graph cocycle flows coboundaries over $d=2$, i.e.\ $Q_{d=2}^{\gamma_{2\ell+1}}=\lshad P(\varrho),\vec{X}_{d=2}^{\gamma_{2\ell+1}}\rshad$, but also there do exist particular solutions $X_{d=2}^{\gamma_i}$ that conjecturally provide linear combinations of (Kontsevich) micro-graphs over which solutions $\vec{X}_{d\geqslant 2}^{\gamma_i}$ appear in higher dimensions (e.g., for $\gamma_3$ and $d=4$, see~\cite{MJB}).

\begin{prop}[{\cite{Ascona96,Anass2017}}]\label{PropG3triv2D}
For the tetrahedron~$\gamma_3$ on $n=4$ vertices, the trivialising vector field $\vec{X}^{\gamma_3}_{d=2}(P\otimes P\otimes P)$ is unique modulo Hamiltonian vector fields with~$H(P\otimes P)$ given by Kontsevich graphs. The formula of a particular representative $\vec{X}^{\gamma_3}_{d=2} \mod\vec{X}_{H(P\otimes P)}$ is encoded by the `sunflower' graph (see~\cite[App.\,F]{f16} and~\cite{skew23}),
\begin{equation}\label{EqSunflower}
X^{\gamma_3}_{d=2} = (0,1;1,3;1,2)+2\cdot(0,2;1,3;1,2) = \Gamma'+2\Gamma'' =
\raisebox{0pt}[6mm][4mm]{\unitlength=0.4mm
\special{em:linewidth 0.4pt}
\linethickness{0.4pt}
\begin{picture}(17,24)(5,5)
\put(-5,-7){
\begin{picture}(17.00,24.00)
\put(10.00,10.00){\circle*{1}}
\put(17.00,17.0){\circle*{1}}
\put(3.00,17.0){\circle*{1}}
\put(10.00,10.00){\vector(0,-1){7.30}}
\put(17.00,17.00){\vector(-1,0){14.00}}
\put(3.00,17.00){\vector(1,-1){6.67}}
\bezier{30}(3,17)(6.67,13.67)(9.67,10.33)
%
%
\put(17,17){\vector(-1,-1){6.67}}
\bezier{30}(17,17)(13.67,13.67)(10.33,10.33)
\bezier{52}(17.00,17.00)(16.33,23.33)(10.00,24.00)
\bezier{52}(10.00,24.00)(3.67,23.33)(3.00,17.00)
\put(16.8,18.2){\vector(0,-1){1}}
\put(10,17){\oval(18,18)}
\put(10,10){\line(1,0){10}}
\bezier{52}(20,10)(27,10)(21,16)
\put(21,16){\vector(-1,1){0}}
\end{picture}
}\end{picture}}
\:.  
\end{equation}
$\bullet$\quad The formula $\vec{X}_{d=2}^{\gamma_3}\in\mathfrak{X}^1(\mathbb{R}^2)$ of this `sunflower' vector field is Hamiltonian (in the classical sense) with respect to the standard symplectic structure $\omega_2=\Id x\wedge\Id y$ on~$\mathbb{R}^2$: one can readily inspect that 
\[
\vec{X}^{\gamma_3}_{d=2}= \bigl(\Id x\wedge \Id y\bigr)^{-1}
 \bigl( \Id_{\text{deRham}} \text{Ham}^{\gamma_3}(P\otimes P\otimes\varrho)  \bigr),
\]
where the formula of the Hamiltonian $\text{Ham}^{\gamma_3}$ is encoded again by a graph, namely (1,3;1,2): it is built over three vertices 1,2,3 from two wedges (with tops 2 and 3); the vertex 1 is terminal, it contains $\varrho$.\\[0.5pt]
$\bullet$\quad Under the mapping of $\Id(\text{Ham}^{\gamma_3})$ by $(\Id x\wedge \Id y)^{-1}$ to the vector field $\vec{X}_{d=2}^{\gamma_3}$, two edges are issued from that terminal vertex: one edge goes to the sink 0 (that is, to the argument of the 1-vector), while the other edge works by the Leibniz rule over all vertices of the Hamiltonian graph Ham$^{\gamma_3}$, whence at least one tadpole arises (specifically, in the graph $\Gamma'$ of the `sunflower').
\end{prop}

\begin{rem}[{\cite{skew23}}]
In any solution $\smash{\vec{X}^{\gamma_3}_{d=2}}$, at least one tadpole is necessary.
\end{rem}

For the graph cocycles $\gamma\in\bigl\{ \gamma_5$,\ $\gamma_7$,\ $[\gamma_3,\gamma_5] \bigr\}$ beyond $\gamma_3$, the deformations $\dot{P}=\Or(\gamma)([P])$ of bi-vectors $P$ over $\mathbb{R}^2$ are trivialised by vector fields~$\vec{X}^\gamma_{d=2}([P])$, also encoded by Formality graphs; 
in each case, there is a particular solution $\vec{X}_{d=2}^{\gamma}=\omega_2^{-1}\bigl( \Id_{\text{dR}}(\text{Ham}^{\gamma})\bigr)$ with the Hamiltonian encoded by graphs.

\begin{state}\label{PropG5triv2D} 
For the pentagon-wheel graph cocycle $\gamma_5$ (see~\cite{JNMP2017,sqs17}), the respective flow $\dot{P}=\Or(\gamma_5)(P_{d=2}^{\otimes^6})\mathrel{{=}{:}}Q_{d=2}^{\gamma_5}$ on the space of bi-vectors~$P$ over $\mathbb{R}^2$ is a coboundary, $Q_{d=2}^{\gamma_5}=\lshad P,\vec{X}_{d=2}^{\gamma_5}\rshad$, with respect to the vector field $\vec{X}_{d=2}^{\gamma_5}=\omega_2^{-1}(\Id_\text{dR}(\text{Ham}^{\gamma_5}))$ built of five wedges.\\[0.5pt]
$\bullet$\quad  The Hamiltonian Ham$^{\gamma_5}$ is built of four wedges --\,their tops in the vertices 2,3,4,5\,-- and one terminal vertex~1 (containing $\varrho(x,y)$) in each of its three graphs; its encoding is 
\[
\text{Ham}^{\gamma_5}= 6\cdot[3,5; 4,5; 1,2; 1,2]
 - 2\cdot[3,5; 4,5; 1,2; 1,4]
 - 2\cdot[1,5; 1,4; 1,2; 1,3].
\] 
$\bullet$\quad  The trivialising vector field $\smash{\vec{X}_{d=2}^{\gamma_5}}=\phi(X_{d=2}^{\gamma_5})$ is encoded by $15=3\times5$ graphs built of five wedges over one sink~0 in each term: the encodings for each of the three $5$-\/tuples of $1$-\/vector graphs are obtained by issuing the wedge from vertex~1, namely by sending its Left arrow to the sink~0 and by letting the Right edge (1,$i$) run over all the aerial vertices $i\in\{1,2,3,4,5\}$ (so that the edge (1,1) is the tadpole).\\[0.5pt]
$\bullet$\quad  This solution $\vec{X}_{d=2}^{\gamma_5}$ of the $\gamma_5$-flow trivialisation problem $Q_{d=2}^{\gamma_5}=\lshad P,\smash{\vec{X}_{d=2}^{\gamma_5}}\rshad$ over $d=2$ is unique modulo Poisson vector fields $\vec{X}_H=\lshad P,H\rshad$ with Hamiltonians $H(P^{\otimes^4})$ encoded by Kontsevich graphs on four wedges.
\end{state}


\begin{rem}\label{RemG7G35triv2D}
For the heptagon-wheel graph cocycle $\gamma_7$ from~\cite{JNMP2017}, a solution $\vec{X}_{d=2}^{\gamma_7}=\omega_2^{-1}(\Id_{\text{dR}}\text{Ham}^{\gamma_7})$ of the trivialisation problem $\Or(\gamma_7)(P^{\otimes^8}_{d=2})=\lshad P,\vec{X}_{d=2}^{\gamma_7}\rshad$ over $\mathbb{R}^2$ is known from~\cite[\S6.4]{BuringPhD}.

The graph commutator $[\gamma_3,\gamma_5]$, itself not a $(2\ell+1)$-wheel generator of $\mathfrak{grt}$, acts on bi-vectors $P$ over $\mathbb{R}^2$ in a similar way: 
$\Or\bigl( [\gamma_3,\gamma_5]\bigr)(P_{d=2}^{\otimes^9})=\lshad P,\vec{X}_{d=2}^{[\gamma_3,\gamma_5]}\rshad$ with $\vec{X}_{d=2}^{[\gamma_3,\gamma_5]}=\omega_2^{-1}(\Id_{\text{dR}}\text{Ham}^{[\gamma_3,\gamma_5]})$ built of eight wedges.\footnote{R.~Buring, private communication (14~May 2024).}

However, in both the cases (for $\gamma_7$ and for $[\gamma_3,\gamma_5]$), the respective Hamiltonians, referred to the standard symplectic structure $\omega_2$ on $\mathbb{R}^2$, were obtained at the level of homogeneous differential polynomials in $\varrho$, that is, not yet at the level of Formality graphs built only of wedges and one terminal vertex --- in contrast with Propositions~\ref{PropG3triv2D} and~\ref{PropG5triv2D} where we make that graph realisation explicit.
\end{rem}

\begin{open}\label{OpenPrbDeRhamHamX}
Is it true that in dimension $d=2$, for each graph cocycle~$\gamma$ from the Grothendieck\/--\/Teich\-m\"ul\-ler Lie algebra~$\mathfrak{grt}$ generated by the $(2\ell+1)$-\/wheel cocycles~$\gamma_{2\ell+1}$, the $\gamma$-\/flow trivialisation problem always has a solution of the shape
$\smash{\vec{X}^{\gamma}_{d=2}} = \omega_2^{-1} \bigl( \Id \text{Ham}^\gamma \bigr)$,
where, moreover, the directed graphs in the $0$-\/vector $\text{Ham}^\gamma$ are built of $2\ell$~wedges (for copies of~$P$) and one terminal vertex (with~$\varrho(x,y)$)\,?
\end{open}

\begin{rem}\label{RemSunflowerGivesGraphs3D4D}
The `sunflower' graph~\eqref{EqSunflower} is special: on its $(d=3,4)$-\/descendants, i.e.\ on the set of Kontsevich micro\/-\/graphs which appear from Kontsevich's two graphs in the `sunflower', there exist a solution in dimension three and a solution in dimension four (see~\cite{MJB}).
In the subsequent paper~\cite{FS}, by running over the synonyms of either graph in the `sunflower' solution of the trivialisation problem at~$d=2$, we detect that this effect is not generic: over $(d\geqslant3)$-descendants of the synonyms, solutions typically cease to exist.
\end{rem}



\section{Kontsevich graph flows of Nambu\/--\/Poisson brackets over~$\mathbb{R}^{\geqslant 3}$}\label{SecFlowsHighDim}

The Kontsevich graph cocycles $\gamma$ on $n$ vertices and $2n-2$ edges act on the space $\mathfrak{X}^2(\mathbb{R}^2)$ of bi-vectors over affine spaces $\mathbb{R}^d$ of any dimension $d\geqslant 2$; the graph flows $\dot{P}=\Or(\gamma)(P^{\otimes^n})$ preserve the subset of all \textsl{Poisson} bi-vectors~$P$ satisfying the Jacobi identity $\tfrac{1}{2}\lshad P,P\rshad=0$. Let us study whether in the set of \textsl{all} Poisson bi-vectors, Kontsevich's graph flows preserve the class $\{P(\varrho,[\boldsymbol{a}])\}$ of Nambu-determinant Poisson structures on $\mathbb{R}^d$.


\begin{define}\label{DefNambuClassIsPreserved}
The class of Nambu\/--\/Poisson brackets $P(\varrho,[\boldsymbol{a}])$ on $\mathbb{R}^d$, $d\geqslant 3$, is preserved by a flow $\tfrac{d}{d\varepsilon}(P)=Q([P])$
if there exist, for all $\varrho\cdot\partial_{\boldsymbol{x}}\in\mathfrak{X}^d(\mathbb{R}^d)$ and Casimirs $a_i\in C^{\infty}(\mathbb{R}^d)$ simultaneously, the evolution equations $\tfrac{d}{d\varepsilon}(\varrho)=R([\varrho],[\boldsymbol{a}])$ and $\tfrac{d}{d\varepsilon}(a_i)=A_i([\varrho],[\boldsymbol{a}])$ such that the evolution of Nambu bi-vector~$P\bigl(\varrho,[\ba]\bigr)$ along $Q([P])$ amounts to the Leibniz rule for $d/d\varepsilon$ acting on its components:
\begin{equation}\label{EqEvolAviaG}
\tfrac{d}{d\varepsilon}\Bigl(P\bigl([\varrho],[\ba]\bigr)\Bigr) = Q([P]) = P\bigl(\tfrac{d}{d \varepsilon}{\varrho},[\ba]\bigr) +
\sum\nolimits_{i=1}^{d-2} P\bigl(\varrho,[a_1],\ldots,[\tfrac{d}{d \varepsilon} a_i],\ldots,[a_{d-2}]\bigr),
\end{equation}
that is, the evolutions of~$P(\varrho,[\ba])$ and of its elements, $\varrho$ and Casimirs~$a_i$,
match.
\end{define}

\begin{ex}[\cite{skew21}]\label{ExGFlowRestrictsToNambu}
    The $\gamma_3$-deformation restricts to the class of Nambu-determinant Poisson bi-vectors on (at least) $\mathbb{R}^3$ and~$\mathbb{R}^4$. The same is true also for the graph cocycle $\gamma_5$ and its action on the Nambu\/--\/Poisson class of brackets~\eqref{EqNambuCoord} over~$\mathbb{R}^3$.
\end{ex}

\begin{conjecture}[see \cite{skew21}]\label{ConjEvolARho}
Consider the Kontsevich $\gamma$-cocycle deformation $\dot{P}=\Or(\gamma)(P^{\otimes^n})$, where $n$~is the number of vertices in each term of $\gamma$ and $2n-2$ is the number of edges, and assume that this flow $\dot{P}=Q^{\gamma}([P])$ does restrict to the flow $Q_d^{\gamma}$ on the class of Nambu\/--\/Poisson bi\/-\/vectors $P\bigl(\varrho,[\ba]\bigr)$ over~$\BBR^d$ for some $d\geqslant 3$. By definition, put (with reference of tuples of arguments to vertices of each term in the graph cocycle~$\gamma$):
\[
\tfrac{d}{d \varepsilon} a_i = \Or(\gamma)(a_i\otimes P^{\otimes^{n-1}}) +
 \Or (\gamma)(P\otimes a_i\otimes P^{\otimes^{n-2}}) +
  \ldots +
 \Or (\gamma)(P^{\otimes^{n-1}}\otimes a_i).
\]
Then, the conjecture is that the fraction,
\[
\tfrac{d}{d \varepsilon} \varrho = \frac{ \Bigl( Q_d^{\gamma}([P]) 
  - \sum_{i=1}^{d-2} P\bigl(\varrho,[a_1],\ldots,[\tfrac{d}{d \varepsilon} {a_i}],\ldots,[a_{d-2}]\bigr)\Bigr) (f,g) }
{ \det\Bigl( \dd \bigl( f,g,a_1,\ldots,a_{d-2}\bigr) \big/ \dd\bigl(x^1,\ldots,x^d\bigr) \Bigr) },
\]
is differential polynomial in~$\varrho$ and~$a_i$, 
so that Leibniz rule~\eqref{EqEvolAviaG} tautologically holds.
\end{conjecture}

\begin{ex}\label{ExAdotRhoDotFromG}
    The above conjecture is confirmed to be true for the tetrahedron graph cocycle $\gamma_3$ and dimensions $d=3,4$, and for the pentagon-wheel graph cocycle $\gamma_5$ and Nambu\/--\/Poisson structures over $d=3$.
\end{ex}


In the rest of this section we discuss the observed trivialisation of Kontsevich's graph cocycle flows $\dot{P}(\varrho,[\ba])=Q_d^{\gamma}([P])$ in the Lichnerowicz\/--\/Poisson second cohomology w.r.t.\ $\lshad P(\varrho,[\ba]),\cdot\rshad$, that is, we recall some evidence for the existence of vector field solutions $\vec{X}_d^{\gamma}([\varrho],[\ba])$ for the equations $Q_d^{\gamma}([P])=\lshad P,\vec{X}_d^{\gamma}\rshad$. (The known solutions $\vec{X}_{d=3,4}^{\gamma_3}$ are encoded by the Nambu micro-graphs but not by Kontsevich micro-graphs, as they do not stem directly from the previously known solutions $\vec{X}_{d=2}^{\gamma_3}$, see papers~\cite{MJB,FS}.)

\begin{lem}\label{LemmaNambuSkew}
In any dimension $d\geqslant3$, Nambu\/--\/Poisson bi-vectors in~\eqref{EqNambuCoord} are odd 
w.r.t.\ every Casimir $a_i$, namely 
$P(\varrho,\ldots,[-a_i],\ldots) = -P(\varrho,\ldots,[a_i],\ldots)$.\\[0.5pt]
$\bullet$\quad For all 
$d\geqslant4$, Nambu\/--\/Poisson bi-vectors in~\eqref{EqNambuCoord} are totally antisymmetric w.r.t.\ permutations $\sigma\in \mathbb{S}_{d-2}$ of the set of Casimirs $\ba$: we have 
$P(\varrho,[\sigma(\ba)]) = (-)^{\sigma}\cdot P(\varrho,[\ba])$.
\end{lem}

\begin{rem}\label{RemGeneratorsGsymmetric}
    In contrast with the above lemma, the graph cocycle generators $\gamma_{2\ell+1}$ of $\mathfrak{grt}$ consist of $(2\ell+1)$-wheels and other graphs on $2\ell+2$ vertices; this number is even, whence 
$Q_{d\geqslant 3}^{\gamma_i}(P(\varrho,\ldots,[-a_i],\ldots))\equiv Q_{d\geqslant 3}^{\gamma_i}(P(\varrho,\ldots,[a_i],\ldots))$ and likewise, $Q_{d\geqslant 4}^{\gamma_i}(P(\varrho,\sigma(\ba)))\equiv Q_{d\geqslant 4}^{\gamma_i}(P(\varrho,[\ba]))$ for all $\sigma\in \mathbb{S}_{d-2}$. (The reasoning does not work for $[\gamma_3,\gamma_5]$ on 9 vertices and for other (iterated) commutators on an odd number of vertices.) This reveals that Kontsevich's graph cocycles $\gamma_i$, acting on Nambu\/--\/Poisson bi-vectors by infinitesimal deformations $\dot{P}=Q_d^{\gamma_i}([P])$, at once lose the structural property of these brackets.
\footnote{%
In fact, the Kontsevich deformation bi-vectors $Q_d^{\gamma_i}$ are well defined for the symplectic foliation of~$\mathbb{R}^d$, no matter how it is described by the level sets of the Casimirs~$a_i$ (or of~$-a_i$) or of their permutations, because the level sets of linear combinations~$A\ba$ define the same loci if~$\det(A)\neq0$.}
\end{rem}

This loss of structural property of~$P(\varrho,[\ba])$ by~$Q_d^{\gamma_i}([P])$ is an indirect but strong evidence that these graph cocycle flows are coboundaries over all $d\geqslant 3$ for each~$\gamma_i$.

\begin{ex}[see~\cite{skew23,MJB}]\label{ExG3trivDim3Dim4}
    The tetrahedral flow $\dot{P}=Q_d^{\gamma_3}([P])$ is a Poisson coboundary for the class of Nambu\/--\/Poisson brackets~\eqref{EqNambuCoord} over $d=3$ and~$d=4$.
\end{ex}

\begin{state}\label{PropCleverSystem}
Along any vector field $\vec{Y}=-\vec{X}\in\mathfrak{X}^1(\mathbb{R}^{d\geqslant 3})$ on~$\BBR^d$, the scalar functions $a_i$ evolve as fast as $\dot{a_i}=(-\vec{X})(a_i)$, and the evolution of $\varrho\cdot\partial_{\boldsymbol{x}}\in\mathfrak{X}^d(\mathbb{R}^d)$ is $\dot{\varrho}\,\partial_{\bx}=\lshad \varrho\partial_{\bx},\vec{X}\rshad$, which is standard. Now, the found vector fields $\vec{X}_d^{\gamma_3}$ trivialising the tetrahedral $\gamma_3$-flows of Nambu brackets~\eqref{EqNambuCoord} over $\mathbb{R}^3$ and $\mathbb{R}^4$ are such that
\begin{equation}\label{EqAdotRhoDotX}
\tfrac{d}{d \varepsilon} a_i = \bigl(-\vec{X}^{\gamma_3}_{d}\bigr)(a_i)
\qquad \text{and} \qquad 
\tfrac{d}{d \varepsilon}(\varrho)\,\dd_{\bx} = \lshad \varrho\,\dd_{\bx}, \vec{X}^{\gamma_3}_{d} \rshad.
\end{equation}
In other words, the evolution of Casimirs, obtained directly from the graph cocycle~$\gamma_3$ (see Conjecture~\ref{ConjEvolARho} and Example~\ref{ExAdotRhoDotFromG}), and the evolution of $d$-\/vector $\varrho\,\dd_{\bx}$, read from the $\gamma_3$-deformation of the Nambu bi\/-\/vector~$P\bigl(\varrho,[\ba]\bigr)$, agree with the law of evolution of zero-{} and $d$-vectors along the vector field which trivialises the $\gamma_3$-flow.\footnote{\label{FootVFaDotRhoDot}%
At $d=3$, equality~\eqref{EqAdotRhoDotX} is inspected after the trivialising vector field $\vec{X}^{\gamma_3}_{d=3}$ is found by solving the equation $Q^{\gamma_3}_{d=3}([P]) = \lshad P, \vec{X}^{\gamma_3}_{d=3}([P]) \rshad$.
At $d=4$, equations~\eqref{EqAdotRhoDotX} are solved for~$\vec{X}^{\gamma_3}_{d=4}$, and then the equality $Q^{\gamma_3}_{d=4}([P]) = \lshad P, \vec{X}^{\gamma_3}_{d=4}([P]) \rshad$ is confirmed.
Likewise, for the pentagon\/-\/wheel cocycle~$\gamma_5$, Eqs~\eqref{EqAdotRhoDotX} over~$\BBR^3$ are solved first.}
\end{state} 





   



\noindent\textbf{Conclusion.}\quad
For the infinitesimal deformations $\dot{P} = Q([P])$ of Poisson bi\/-\/vectors~$P$, the calculus of multivectors using Kontsevich and Nambu (micro-)\/graphs turns the PDE problem of deformations' (non)\/triviality in the Poisson cohomology into a problem from linear algebra. Yet the evaluation map~$\phi$, acting from (Nambu micro-)\/graphs $\Gamma$ to poly\/-\/linear polydifferential operators and then, by antisymmetrisation, to multivectors $\phi(\Alt(\Gamma))\in\mathfrak{X}^k(M^d_{\text{aff}})$, does have a nontrivial kernel, whence stem vanishing graphs, synonyms, and longer linear relations $\phi\bigl(\sum_i c_i \Gamma_i \bigr)=0$. The formulas which $\phi$~produces from (micro-)\/graphs are well defined w.r.t.\ affine changes $\mathbf{x}'(\bx) \rightleftarrows \bx(\mathbf{x}')$ locally on~$M^d_{\text{aff}}$. We pose the problem of effective description of multivector\/-\/valued invariants of the affine (essentially, only of $GL(d)$) group action on tensor fields over~$M^d_{\text{aff}}$, so that the new kernel is as small as possible.

It remains unclear why the Nambu class $\bigl\{ P(\varrho,[\ba]) \bigr\}$ is preserved by Kontsevich's graph cocycles~$\gamma$, namely why the cocycles~$\gamma$ yield the genuine evolution of Casimirs~$\ba$ and why the evolution of~$\varrho\,\dd_{\bx}$ is then well defined from~$Q^\gamma_d([P])$. The underlying mechanics of cross\/-\/terms cancellation looks similar to the noted preservation of identities $\phi\bigl(\sum_i c_i \Gamma_i \bigr)=0$ by the graph embeddings $\Gamma\hookrightarrow\smash{\widehat{\Gamma}}$ into dimension~$d+1$. By understanding the nature of (both) the mechanism(s), we shall gain deeper insight into the algebra and combinatorial topology of $GL(d)$-\/invariants.

We see that 
Nambu\/--\/Poisson brackets resist the Kontsevich graph action. We detect that the flows $\dot{P} = Q^\gamma_d([P(\varrho,[\ba])])$ are Poisson coboundaries: $Q^\gamma_d = \lshad P, \vec{X}^\gamma_d \rshad$, but the vector fields $\smash{\vec{X}^\gamma_{d\geqslant 3}}$ are not obtained from $d=2$ by mere expansion of Leibniz rules. The choice of Nambu micro\/-\/graphs for a solution $\smash{\vec{X}^\gamma_d}$ to appear is not yet codified; our preference of \textsl{the most natural} Ansatz for $\smash{\vec{X}^{\gamma_3}_{d=3,4}}$ in~\cite{MJB} is intuitive.


\medskip{\small 
\noindent\textbf{Acknowledgements.}\quad The authors thank the organisers of the international conference on Integrable Systems \&\ Quantum Symmetries (ISQS28) held on 1--5~July 2024 at CVUT Prague, Czech Republic, for stimulating discussions.
The authors are grateful to R.~Buring (INRIA Saclay, France) for the availability of \textsf{gcaops} software and support. The authors thank the Center for Information Technology of the University of Groningen for their support and for providing access to the H\'abr\'ok high performance computing cluster. The authors thank the University of Groningen for partial financial support.

A part of this research was done while AVK was visiting at the $\smash{\text{IH\'ES}}$, supported in part by the Nokia Fund. AVK~thanks the $\smash{\text{IH\'ES}}$ for hospitality, and thanks M.~Kontsevich for helpful discussions and advice.

}

\end{document}